\documentclass[11pt]{amsart}
\usepackage{amsmath,amsthm,amssymb,amsxtra,enumerate,bbm,xcolor}
\usepackage[colorlinks=true,linkcolor=blue, citecolor=red]{hyperref}
\usepackage[nameinlink,capitalize]{cleveref}
\crefname{equation}{}{}
\numberwithin{equation}{section}
\allowdisplaybreaks
\newtheorem{lemma}{Lemma}[section]
\newtheorem{proposition}[lemma]{Proposition}
\newtheorem{theorem}[lemma]{Theorem}
\newtheorem{corollary}[lemma]{Corollary}

\theoremstyle{remark}
\newtheorem{definition}[lemma]{Definition}
\newtheorem{remark}[lemma]{Remark}
\newtheorem{assumption}{Assumption}

\newcommand{\one}{{\rm 1\hspace{-0.1cm}I}}
\newcommand{\oM}{{M+1}}
\newcommand{\cB}{\mathcal{B}}
\newcommand{\cF}{\mathcal{F}}
\newcommand{\cG}{\mathcal{G}}
\newcommand{\cH}{\mathcal{H}}

\newcommand{\DD}{{\mathbb{D}}}
\newcommand{\PP}{{\mathbb{P}}}
\newcommand{\G}{{\mathbb{G}}}
\newcommand{\R}{{\mathbb{R}}}

\newcommand{\E}{{\mathbb{E}}}
\newcommand{\N}{{\mathbb{N}}}

\newcommand{\bS}{\sigma}
\newcommand{\tnm}{t_{n-1}}
\newcommand{\tn}{t_n}

\newcommand{\sptext}[3]{\hspace{#1 em}\mbox{#2}\hspace{#3 em}}
\newcommand{\card}[1]{\rm card(#1)}
\newcommand{\de}{{\delta}}

\newcommand{\Om}{{\Omega}}
\newcommand{\vare}{\varepsilon}
\newcommand{\s}{{\sigma}}
\newcommand{\al}{{\alpha}}
\newcommand{\be}{{\beta}}
\newcommand{\vph}{{\varphi}}
\newcommand{\var}{{\rm var}}
\newcommand{\half}{\frac{1}{2}}
\newcommand{\brac}{\left(}
\newcommand{\kets}{\right)}
\newcommand{\kla}{\left ( }
\newcommand{\mer}{\right ) }
\newcommand{\klas}{\left [ }
\newcommand{\mers}{\right ] }
\newcommand{\noo}{\left \|}
\newcommand{\rrm}{\right \|}
\newcommand{\bet}{\left |}
\newcommand{\rag}{\right |}
\newcommand{\ra}{\rightarrow}
\newcommand{\equa}{\begin{eqnarray*}}
\newcommand{\tion}{\end{eqnarray*}}
\newcommand{\probsp}{(\Omega, \cF,\PP)}
\newcommand{\pl}{\ \ }
\newcommand{\pll}{\ \ \ \ }
\newcommand{\dom}{{\rm Dom}}

\begin{document}

\title[]{Approximation of certain stochastic integrals with
         anticipating integrands}

\author{Hannah Geiss \and Stefan Geiss \and Onni Hinkkanen}
\address{
        Department of Mathematics and Statistics\\
        University of Jyv\"{a}skyl\"{a} \\
        Ahlmaninkatu 2 \\
        Mattilanniemi (MaD) \\
        FI-40100 University of Jyväskylä \\
        Finland \\
        hannah.r.geiss@maths.jyu.fi\\
        geiss@maths.jyu.fi\\ 
        onni.u.i.hinkkanen@jyu.fi}

\date{\today}

\begin{abstract}
We study the quantitative approximation of certain stochastic integrals, 
where we use discrete time approximations under initial enlargement of 
filtration. It turns out that the approximation rate is in general the same
as in the case of no additional information, however, the asymptotic
constant improves. 
\end{abstract}


\maketitle

\setcounter{tocdepth}{1}
\tableofcontents


\section*{Introduction}

Let $W = (W_t)_{t\in [0,1]}$ be a one-dimensional standard Wiener process 
and either $X=W$ be the Wiener process itself or $X_t:=e^{W_t -\frac{t}{2}}$ be the
geometric Brownian motion.
Assume $I:=\R$ for $X=W$ and $I:=(0,\infty)$ in case $X=S $, and
let $f:I \to\R $ be a  Borel function such that $f(X_1)\in L_2.$
In \cite{Gei18} and  \cite {Gei19} the approximation number 
\begin{equation}\label{a-f-tau}
\pll a_{X}(f|\tau) 
   : = \inf \noo f(X_1) - \E f(X_1) - \sum_{n=1}^N
       g_{t_{n-1}} (X_{t_n}-X_{t_{n-1}}) \rrm_{L_2} 
\end{equation}
is considered where $\tau=\{0=t_0<\cdots <t_N=1\}$ is a given deterministic time net
and the infimum is taken over all $\cF_{t_{n-1}}$-measurable
$g_{t_{n-1}}:\Om\ra \R$ such that
$g_{t_{n-1}} (X_{t_n}-X_{t_{n-1}})$ is square-integrable. The above quantity can be
also written as
\[ a_{X}(f|\tau) = \|  f(X_1) - P_\tau(f(X_1)) \|_{L_2}, \]
where $P_\tau: L_2\to L_2$ is the orthogonal
projection onto the subspace generated by  $x_0 + \sum_{n=1}^N
g_{t_{n-1}} (X_{t_n}-X_{t_{n-1}})$ with $x_0\in\R$ and under
the above conditions on the random variables $g_{t_{n-1}}$. If $X$ is the
geometric Brownian motion, then in Stochastic Finance $a_{X}(f|\tau)$ stands for
the minimal quadratic risk of a portfolio in the
discounted Black-Scholes model adjusted at the time points
$t_0,...,t_{N-1}$. 
\medskip

It is well known \cite{Gobet-Temam, Gei19} that
the behaviour of $a_{X}(f|\tau)$ depends on the smoothness of $f$.
For example, for the equidistant nets $\tau_N =(n/N)_{n=0}^N$ it was
shown (see \cite{Gei18,Gei19}) that
\[ \sup_N \sqrt{N} a_X(f|\tau_N) < \infty
   \sptext{1}{if and only if}{1}
   f(X_1) \in \DD_{1,2}, \]
where $\DD_{1,2}$ denotes the Malliavin Sobolev space. 
In addition, one has that
\[   \lim_N  N \, a_X^2(f|\tau_N)
   = \frac{1}{2} \int_0^1
     \E \left | \s^2(X_t)  \frac{\partial^2 F}{\partial x^2}(t,X_t) \right |^2 dt, \]
where   
\[ F(t,x) := \E_{X_t=x} f(X_1), \]
so that $F(1,x) = f(x)$ and
\begin{equation}\label{diffeq-F}   
     \frac{\partial F}{\partial t}
   + \frac{\s^2}{2} \frac{\partial^2 F}{\partial x^2} = 0 
\end{equation}
on $(-\vare,1)\times I$ for some $\vare>0$. 
\medskip

For a general $f$ with $f(X_1)\in L_2$, which is not almost surely
linear, one has always the lower bound
\[  \inf_{\card{\tau}=N+1} a_X(f|\tau) \ge \frac{\delta}{\sqrt{N}} \]
for some $\delta>0$ depending on $f$ (see \cite{Gei20}).
\medskip

The techniques used to obtain the above results are mainly based on 
martingale methods. They do not extend to the anticipative 
setting, which we will consider in this article, where the  $g_{t_{n-1}}$ in
\eqref{a-f-tau} will be allowed to be $\cG_{t_{n-1}}$-measurable and $x_0$
is a $\cG_0$-measurable random variable
assuming $(\cG_t)_{t\in [0,1]}$  is a filtration such that
$\mathcal{G}_t :=  \cF_t \vee \sigma(L)$ for
a random variable $L.$ The approximation number becomes

\begin{equation}\label{a-f-tau-G}
  \pll a_{X}(f|\tau, \mathcal{G}) 
    := \inf \noo f(X_1) - \Phi(L) 
       - \sum_{n=1}^N g_{t_{n-1}} (X_{t_n}-X_{t_{n-1}}) \rrm_{L_2}, 
\end{equation}
where the infimum is taken over all $g_{t_{n-1}}$ which are
$\cG_{t_{n-1}}$-measurable such that $g_{t_{n-1}} (X_{t_n}-X_{t_{n-1}})$ is square-integrable,
and over all Borel functions $\Phi$ with $\Phi(L) \in L_2$.
\medskip

Stochastic models under insider information have been studied in the literature, see for example 
\cite{Biagini:Oksendahl:2005,Campi:2005,Oksendahl:Rose:2017,Elizalde:Escudero:2022,Elizalde_etal:2025}.
A natural approach to study insider information is to use Malliavin calculus and Skorohod integration.
\medskip

If $X$ is the geometric Brownian motion, the setting we investigate can be understood as discrete time hedging in the
Black-Scholes model under insider information given by a vector $L$. The purpose of this article is to get a first insight into
quantitative effects of additional information on discrete time hedging.
\medskip

The article is organised as follows. 
After stating some preliminaries in \cref{sec:preliminaries}, we describe in \cref{sec:results} the main results.
We continue with examples in \cref{sec:examples}. In \cref{sec:FRC_KW} we consider the {\em first running chaos} 
and the {\em Kunita Watanabe projection}, both will be used to prove the main results
in \cref{sec:proof_main_results}.


\section{Preliminaries}
\label{sec:preliminaries}

Let $W = (W_t)_{t\in [0,1]}$ be a $1$-dimensional standard Wiener process
defined on the probability space $\probsp$, where $\cF$ is the completion of
$\sigma (W _t: t\in [0,1])$ and $(\cF_t)_{t\in [0,1]}$ the
augmented natural filtration of $W$. W.l.o.g. we assume that all paths
of $W$ are continuous and that $W_0\equiv 0$.
In the article we let $X=(X_t)_{t\in [0,1]}$ be either the Wiener process $W$ or the geometric Brownian motion
$(S_t)_{t\in [0,1]}$,
\[ S_t := e^{W_t-\frac{t}{2}}, \]
so that
\begin{equation}\label{equation:dX_t} 
 dX_t  =  \s(X_t)dW_t \sptext{1}{with}{1}
    \s(x)
  := \begin{cases}
          x &  \textrm{for } X = S, \\
          1 &  \textrm{for } X = W.
      \end{cases}
\end{equation}
Given $n\ge 1$ and an open set $\emptyset\not = G\subseteq \R^n$ , the space of all infinitely many times
differentiable functions $f:G\to \R$, such that all partial derivatives have at most polynomial growth, is denoted  by $C^\infty_p(G)$. Finally $L_2$ stands usually for $L_2(\Omega,\cF,\PP)$.


\section{Main results}\label{sec:main_results}
\label{sec:results}


\subsection{Initial information and approximation problem}
\label{sec:initial_information_approximation}
To generate the initial enlargement $\G=(\cG_t)_{t\in [0,1]}$ of the filtration $(\cF_t)_{t\in [0,1]}$
we assume
\medskip

\begin{enumerate}
\item $0=s_0 < s_1 < \cdots < s_M =:T \le s_\oM := 1$,
\item $v_k\in C_p^\infty(\R^M)$ for $k=1,\ldots,K$,
\item $L_k:= v_k(X_{s_1},...,X_{s_M})$,
\item $L=(L_1,\ldots,L_K)$,
\item $\cG_t:= \cF_t \vee \sigma(L)$.
\end{enumerate}
\medskip

The approximation problem is considered for {\em deterministic} time-nets, but under the {\em enlarged} filtration.
To shorten the notation we will write
\[ f(X):= f(X_{s_1},\ldots,X_{s_\oM}).\]

\begin{definition}
\label{definition:approximation}
For $X$ as given in \eqref{equation:dX_t},
a Borel function $f:I^{\oM}\to \R$ with
$\E |f(X)|^2<\infty$, and a time-net
\[ \tau = \{0 = t_0 < t_1 < \dots < t_N=1\} \]
we define the {\em $L_2$-approximation number of $f$ under the filtration $\G$} by
\begin{equation} \label{eqn1:a-f-tau-G}
      \pll a_{X}(f|\tau, \G)
      := \inf \noo f(X) - \Phi_{0}(L)
            - \sum_{n=1}^N g_{t_{n-1}} (X_{t_n}-X_{t_{n-1}}) \rrm_{L_2},
\end{equation}
where the infimum is taken over all $g_{t_{n-1}}$ which are $\mathcal{G}_{t_{n-1}}$-measurable random variables such that
$g_{t_{n-1}} (X_{t_n}-X_{t_{n-1}}) \in L_2$, and over all Borel functions $\Phi_{0}:\R^K \to \R$ with $\Phi_{0}(L) \in L_2$.
\end{definition}
\smallskip


\subsection{Admissible representation of $f(X)$}
To handle the approximation problem formulated in \cref{definition:approximation}
we introduce the following class of admissible representations:
\medskip

\begin{assumption}
\label{assumptionA}
We assume that $f(X)$ is given by
\begin{equation}\label{eqn:definition:admissible_representation_f}
f(X) = \Phi_{0}(L) +  \int_{0}^{1} \vph(t, X_t, y) \sigma(X_t)dW_t \biggr|_{y = L}  \mbox{ a.s.},
\end{equation}
where $\varphi:[0,1)\times I \times \R^K \to \R$
is a Borel function obtained as follows in the following two cases: 
\medskip

\begin{enumerate}
\item  \label{ass1} If $T<1$, then there is a $\theta\in (0,1]$ and a polynomially bounded $\Phi_{M+1}:\R\to \R$ such that 
      \[ \int_0^1 (1-t)^{1-\theta} \left | \sigma^2(X_t) \frac{\partial^2 F}{\partial x^2}(t,X_t) \right |^2 dt < \infty \]
      with
      $F_\oM(t,x) := \E (\Phi_{M+1}(X_1)| X_t=x)$.
\item \label{ass2} For $m=1,\ldots,M$ we assume $\Phi_m\in C_p^\infty(I \times \R^K)$.
\item  \label{ass3} We define $\varphi(0,x,y)\equiv 0$ and
      \[ \varphi(t,x,y) 
      := \begin{cases}
      \frac{\partial F_m}{\partial x}(t,x,y)    &: t\in (s_{m-1},s_m], m=1,\ldots,M \\
      \frac{\partial F_{M+1}}{\partial x}(t,x)  &: t\in (T,1), T<1 
         \end{cases}, \]
      where $F_m:[s_{m-1},s_m]\times I\times \R^K \to \R$ is given by
      \[ F_m(t,x,y) := \E (\Phi_m(X_{s_m},y)| X_t=x).\]
\end{enumerate}
\end{assumption}
\smallskip

Item \eqref{ass3}   in \cref{assumptionA}  takes care of two cases: If $T=1$, then insider information is used also at the time 
$s_M= s_{M+1} =1$. In this case we need $C_\infty^p$-regularity assumptions for the whole time interval $(0,1]$. 
If $T<1$, then no insider information is used between $T$ and $1$. In this case  the regularity assumptions can be relaxed on
$(T,1]$ which permits treating also  digital options for example.
\bigskip

If there is no risk of confusion we will use the notation 
\begin{equation}\label{eqn:definition:admissible_representation_Z}
  Z:= \int_{0}^{1} \vph(t, X_t, y) \sigma(X_t)dW_t \biggr|_{y = L}
\end{equation}
for the last term in \eqref{eqn:definition:admissible_representation_f}
as this is the main term we are interested in within our approximation problem,
and use the following convention:
\smallskip

\begin{definition} \rm \label{definition:integrands}
The representation of $Z$ in \eqref{eqn:definition:admissible_representation_Z}
under the conditions  of  \cref{assumptionA}
is called an {\it admissible representation} of $Z$.
\end{definition} 

\begin{remark} Let us comment on  \cref{assumptionA}:
From the definition it follows for $s\in [s_{m-1},s_m]$ with $m=1,\ldots,M$ that
      \[   \int_{s_{m-1}}^s \vph(t, X_t, y) \sigma(X_t)dW_t
         = F_m(s,X_s,y) -  F_m(s_{m-1},X_{s_{m-1}},y) \mbox{ a.s.} \]
      and therefore, a.s.,
      \begin{align*}
          \int_0^1 \vph(t, X_t, y) \sigma(X_t)dW_t\biggr|_{y = L}
      & = \sum_{m=1}^M \left [ F_m(s_m,X_{s_m},L) -  F_m(s_{m-1},X_{s_{m-1}},L)\right ] \\
      &   \quad + \one_{T<1} \left [ F_{M+1}(1,X_1) -  F_{\oM}(T,X_T)\right ]  \\
      & = J(X_{s_1},\ldots,X_{s_\oM})
      \end{align*}
      for a Borel function $J:\R^\oM \to \R$.
\end{remark}
\bigskip


\subsection{The main results}

In this section we state the main results, which will be proved in \cref{sec:proof_main_results}.

\begin{definition}
Let $Z$ have an admissible representation. The Riemann sum under the initial enlargement is
\[  Z_{\tau} := \sum_{n=1}^N  \vph(t_{n-1},X_{t_n-1},L)
               (X_{t_n}-X_{t_{n-1}}) \]
with $\tau = \{0 = t_0 < t_1 < \dots < t_N = 1\}$.
\end{definition}
\medskip

\begin{definition}\label{def:time_nets} \
\begin{enumerate}
\item For $R\in \N$ and $\theta \in (0,1]$ we let $\tau_R^\theta:=\{0=t_{0,R}^\theta< \cdots < t_{R,R}^\theta = 1 \}$ with
       \[ t_{n,R}^{\theta} := 1-\left (1-\frac{n}{R}\right )^{\frac{1}{\theta}}
          \sptext{1}{for}{1} n=0, \ldots, R. \]
\item For $0 \le s_1 < \cdots < s_M \le 1$ as in
      \cref{sec:initial_information_approximation} we let
      \begin{align*}
          \sigma_{R+M}^{\theta} 
      &:= \{ t_{n,R}^{\theta}, n=0, \ldots, R \} \cup \{s_1,...,s_M\} \\
      & = \{0 = s_{0,R}^\theta \leq s_{1,R}^\theta \leq \dots \leq s_{R + M,R}^\theta = 1\}.
      \end{align*}
\end{enumerate}
\end{definition}

As in the net $\sigma_{R+M}^{\theta}$ it may happen that $s_{n,R}^\theta=s_{n+1,R}^\theta$
we agree that in this case we replace $\sigma_{R+M}^{\theta}$ by the net
$\{0=t_0<\cdots<t_N=1\}$ such that $\{t_0,\ldots,t_N\}=\sigma_{R+M}^{\theta}$.
\medskip

Our first approximation concerns the Riemann type approximation
of $Z$:
\medskip

\begin{theorem}[\bf Anticipative approximation]
\label{thm:error_limit}
Assume $Z$ has the admissible representation \eqref{eqn:definition:admissible_representation_Z}
and let $\sigma_N^\theta$, $N>M$, be the time-nets from \cref{def:time_nets}.
Then
\[   \lim_{N\to\infty} N \var \left (Z-Z_{\sigma_N^\theta} \right )
   = \frac{1}{2\theta} \int_0^1 (1-t)^{1-\theta}
     \left\| \sigma^2(X_t)
       \vph_x(t,X_t,L)  \right \|^2_{L_2} dt. \]
\end{theorem}
\bigskip

One can improve the approximation in  \cref{thm:error_limit}
by using the Kunita-Watanabe projection. That this indeed
can improve the approximation can be seen from the proof of 
\cref{thm:error_limit} in \cref{sec:proof:thm:error_limit}. 
\smallskip

The following definition
exploits the representation property
of the Wiener space  \cite[Proposition 3.2, page 199]{Revuz-Yor}:
\medskip

\begin{definition}[Kunita-Watanabe projection]
Let $\xi \in L_2$ be such that
\begin{equation}\label{eq:representation_property}
\xi = \E \xi + \int_0^1 k_t dW_t \quad \mbox{ a.s.},
\end{equation}
where $k$ is progressively measurable and
$\E\int_0^1 k_t^2 dt < \infty$.
Assume $\tau =\{0 = t_0 < t_1 < \dots < t_N = 1\}$ and the process $X$ defined in \eqref{equation:dX_t}.
Then the  Kunita-Watanabe projection (KW-projection) is defined as
\begin{equation} \label{eq:kunita_watanabe}
P_\tau^X \xi := \sum_{n=1}^N g_{n-1} (X_{t_n}-X_{t_{n-1}})
\sptext{.5}{with}{.5}
 g_{n-1}
   :=  \frac{\E_{t_{n-1}} \int_{t_{n-1}}^{t_n} \s(X_t) k_t dt}
            {\E_{t_{n-1}}\int_{t_{n-1}}^{t_n} \s^2(X_t)   dt}.
\end{equation}
\end{definition}
\medskip

\begin{theorem}[\bf Anticipative approximation with KW-projection]
\label{statement:R-approximation_with_KW-projection}
Assume $Z$ has the admissible representation \eqref{eqn:definition:admissible_representation_Z}
and let $\sigma_N^\theta$, $N>M$, be the time-nets from \cref{def:time_nets}.
Then
\begin{multline*}
   \lim_{N \to \infty} N \var \left (Z - Z_{\sigma_N^\theta} - P_{\sigma_N^\theta}^X\left(Z - Z_{\sigma_N^\theta}\right) \right ) \\
 = \frac{1}{2\theta} \int_0^1 (1-t)^{1-\theta}\noo \sigma^2(X_t)\E_ t \vph_x(t,X_t,L)  \rrm^2_{L_2} dt.
 \end{multline*}
\end{theorem}
\bigskip

\cref{statement:R-approximation_with_KW-projection} immediately implies the following
estimate for $a_X(f | \tau_N^\theta,\G)$:

\begin{corollary}
Assume $Z$ has the admissible representation \eqref{eqn:definition:admissible_representation_Z}
and let $\sigma_N^\theta$, $N>M$, be the time-nets from \cref{def:time_nets}.
Then
\[ \limsup_{N\to\infty} Na_X(f | \sigma_N^\theta,\G)
   \leq \inf_\vph \frac{1}{2\theta} \int_0^1 (1-t)^{1-\theta}
        \left \| \sigma^2(X_t) \E_ t \vph_x(t,X_t,L)  \right \|^2_{L_2} dt, \]
where the infimum is taken over all functions $\vph$ such that $Z$ has an admissible representation.
\end{corollary}
\smallskip

There is another approximation which gives the same constant of the limit on the RHS as in
\cref{statement:R-approximation_with_KW-projection} which is based on the following construction:

\begin{definition}
\label{definition:projections}
Let ${\tau} = \{0 = t_0 < t_1 < \dots < t_N = 1\}$.
\begin{enumerate}[(1)]
\item We let 
      \begin{multline*}
      \cH(\tau) := \Bigg \{ \int_0^1 K_u dW_u : 
          (K_u)_{u\in [0,1]} \mbox{ progressively measurable process}, \\
          \E \int_0^1 K_u^2 du < \infty, \\
          K_u \mbox{ is $\cF_{t_{n-1}}$-measurable for } 
          u\in (t_{n-1},t_n] 
                      \Bigg \} .
      \end{multline*}
\item As orthogonal complement we let 
      $\mathcal{R}({\tau}) := (\R \oplus \mathcal{H}({\tau}))^\bot$.
\item The maps  $H_{\tau},R_{\tau}: L_2 \to L_2$ denote the 
      orthogonal projections onto  $\mathcal{H}({\tau})$ and  
      $\mathcal{R}({\tau})$, respectively.
\end{enumerate}
\end{definition}
\medskip

We call $\mathcal{H}({\tau})$ the {\em first running chaos}.
The space  $\mathcal{R}({\tau})$ will be referred to as the  {\em remainder} space. This yields to the orthogonal decomposition
\begin{equation}\label{eq:decomposition_L2}
      L_2 = \R \oplus \mathcal{H}({\tau}) \oplus \mathcal{R}({\tau}).
\end{equation}
We let $H_\tau:L_2\to  \cH(\tau)$ be the orthogonal projection onto $\cH(\tau)$.
With this we get the following statement:
\medskip

\begin{theorem}
[\bf Anticipative approximation with running chaos]
\label{thm:2nd_chaos}
Assume $Z$ has the admissible representation \eqref{eqn:definition:admissible_representation_Z}
and let $\sigma_N^\theta$, $N>M$, be the time-nets from \cref{def:time_nets}.
Then
\begin{multline*}
     \lim_{N\to\infty} N  \var\left (Z - Z_{\sigma_N^\theta} - H_{\sigma_N^\theta}\left(Z - Z_{\sigma_N^\theta}\right) \right )  \\
   = \frac{1}{2\theta} \int_0^1 (1-t)^{1-\theta}
     \noo \sigma^2(X_t)
     \E_ t \vph_x(t,X_t,L)  \rrm^2_{L_2} dt.
\end{multline*}
\end{theorem}
\bigskip

The connection between \cref{thm:2nd_chaos} and \cref{statement:R-approximation_with_KW-projection}
will be established by \cref{statement:FRC_KW} and \cref{lemma:Z-Z_N_variance_new}.


\section{Examples}
\label{sec:examples}


\subsection{Full information at $T\in (0,1)$}
We assume $f:I\to \R$ to be Borel measurable and of at most polynomial growth,
$M:=1$, $s_1:=T$, and $v_1(x)=x$ so that $L=X_T$.
With this we choose  $\Phi_2(x):= f(x)$, $\Phi_1(x,y):= 0$, and $\Phi_0(L):= \E[f(X_1)] + F(T,L)$,
and get
\[ f(X_1)  = \Phi_0(L) + [F_1(T,X_T,y) - F_1(0,X_0,y)] + [F_2(1,X_1) - F_2(T,X_T)].  \]
As integrand $\varphi$ we get
\[ \varphi(t,x,y):= \one_{(T,1)}(t) \frac{\partial F_2}{\partial x}(t,x)\]
and therefore
\begin{align*}
&    \hspace*{-3em} 
     \frac{1}{2\theta} \int_0^1 (1-t)^{1-\theta}
     \noo \sigma^2(X_t)
     \E_ t \vph_x(t,X_t,L)  \rrm^2_{L_2} dt \\
& = \frac{1}{2\theta} \int_0^1 (1-t)^{1-\theta}
     \noo \sigma^2(X_t)
     \vph_x(t,X_t,L)  \rrm^2_{L_2} dt \\
& =  \frac{1}{2\theta} \int_T^1 (1-t)^{1-\theta}
     \noo \sigma^2(X_t)
     \frac{\partial^2 F}{\partial x^2}(t,X_t)  \rrm^2_{L_2} dt
\end{align*}
with $F(t,x):= \E [f(X_1)|X_t=x]$. If $X$ is the geometric Brownian motion, then
this example can be interpreted in the way that we start trading only at time $T$.
\medskip


\subsection{Partial information at $T\in (0,1)$}
\label{sec:partial_information_T<1}

Again we assume $f:I\to \R$ to be Borel measurable and of at most polynomial growth,
$M:=1$, and $s_1:=T$. Assume that $I\ni a<b<\infty$ and a $v_1\in C_\infty^p(\R)$
such that
\[ v_1(x) = \left \{ \begin{array}{rcl}
                  0                          & : & x\le a     \\
                  \mbox{strictly increasing} & : & x\in (a,b) \\
                  1                          & : & x\ge b
            \end{array} \right .. \]
The function $v_1$ provides full information of $X$ inside the interval $[a,b]$,
and whether we are left from this interval or to the right of this interval. Next
we define an approximation of the indicator function of $(a,b)^c$ by the help
of $v_1$: We find a function $H_1\in C_\infty^p(\R)$ such that
      \[ H_1(v_1(x)) = \left \{ \begin{array}{rcl}
                  0                          & : & \mbox{ on } [a+ \vare, b-\vare] \\
                  1                          & : & \mbox{ on }  (a,b)^c            \\
                  \mbox{strictly decreasing} & : & \mbox{ on } [a, a+ \vare]       \\
                  \mbox{strictly increasing} & : & \mbox{ on } [b- \vare,b]        \\
            \end{array} \right ., \]
where $0<2\vare<b-a$. We have the identity
\[ F(T, X_T) = \left[F(T, x) \left(1 - H_1(v_1(x))\right) + F(T, x) H_1(v_1(x))\right] \bigr|_{x=X_T}. \]

Since $v_1$ is invertible on the interval $[a,b]$, we may write the first term as
\[ \widetilde{\Phi}(v_1(x)) = F(T, x) \left(1 - H_1(v_1(x))\right) \]
for a Borel function $\widetilde{\Phi}:[0,1]\to \R$. Now we let
\[ \Phi_1(x,y) := F(T,x) H_1(y) \]
so that, for $L=v_1(X_T)$,
\begin{align*}
     F(T,X_T)
& = \left[F(T, x) \left(1 - H_1(v_1(x))\right) + F(T, x) H_1(v_1(x))\right] \bigr|_{x=X_T} \\
& = \left[\widetilde{\Phi}(y) + F(T,X_T) H_1(y)\right] \bigr|_{y=L} \\
& = [\widetilde{\Phi}(L) + F(0,X_0) H_1(L)] + [F(T,X_T) H_1(L) - F(0,X_0) H_1(L)]   \\
& =: \Phi_0(L) + [F_1(T,X_T,L) - F_1(0,X_0,L)].
\end{align*}
This yields to
\[ \vph(t,x,y) =  \frac{\partial F}{\partial x}(t,x)\left[H_1(y) \one_{[0,T]}(t) + \one_{(T, 1)}(t)\right].\]
We get that
\begin{align*}
&    \hspace*{-3em} 
     \frac{1}{2\theta} \int_0^1 (1-t)^{1-\theta}
     \noo \sigma^2(X_t)
     \E_ t \vph_x(t,X_t,L)  \rrm^2_{L_2} dt \\
& = \frac{1}{2\theta} \int_0^T (1-t)^{1-\theta}
     \noo \sigma^2(X_t)
     \frac{\partial^2 F}{\partial x^2}(t,X_t) 
     \E_t H_1(L) \rrm^2_{L_2} dt \\
& +
     \frac{1}{2\theta} \int_T^1 (1-t)^{1-\theta}
     \noo \sigma^2(X_t)
     \frac{\partial^2 F}{\partial x^2}(t,X_t)  \rrm^2_{L_2} dt
\end{align*}
and
\begin{align*}
&    \hspace*{-3em} 
     \frac{1}{2\theta} \int_0^1 (1-t)^{1-\theta}
     \noo \sigma^2(X_t)
     \vph_x(t,X_t,L)  \rrm^2_{L_2} dt \\
& = \frac{1}{2\theta} \int_0^T (1-t)^{1-\theta}
     \noo \sigma^2(X_t)
     \frac{\partial^2 F}{\partial x^2}(t,X_t) 
     H_1(L) \rrm^2_{L_2} dt \\
& +
     \frac{1}{2\theta} \int_T^1 (1-t)^{1-\theta}
     \noo \sigma^2(X_t)
     \frac{\partial^2 F}{\partial x^2}(t,X_t)  \rrm^2_{L_2} dt.
\end{align*}


\subsection{Partial information at $T=1$}
Assume $K,C,\varepsilon>0$ such that
\[ 0<K<K+\varepsilon<C < C+\varepsilon \]
and assume $v_1,f\in C_\infty^p$ such that
$f(x)=1$ for $x\in (0,K]$, $f(x)=0$ for $x\ge K+\varepsilon$,
$v_1(x)=1$ for $x\in (0,C]$, $v_1(x)=0$ for $x\ge C+\varepsilon$,
and that $f,v_1$ take values in $[0,1]$. Let $X_t=e^{W_t-\frac{t}{2}}$
and $L:=v_1(X_1)$. Then
\[ f(x) = f(x) v_1(x)
   \sptext{1}{and}{1}
   \Phi_1(x,y):= f(x) y \]
so that
\[ \Phi_1(X_1,L) = f(X_1) L = f(X_1) \]
and
\[ \varphi(t,X_t,L) := \frac{\partial F}{\partial x}(t,X_t) L
   \sptext{1}{and}{1}
   \varphi_x(t,X_t,L)  := \frac{\partial^2 F}{\partial x^2}(t,X_t) L.\]
So, in the notation of \cref{sec:partial_information_T<1} we have $H_1(y)=y$.
The interpretation is that $f$ is an approximation of a binary
option with pay-off $\one_{(0,K]}$ and the insider function $v_1$ is an approximation
of the information whether the share price is smaller than or equal to $C$ or not.


\section{First running chaos and Kunita Watanabe projection}
\label{sec:FRC_KW}

\begin{theorem}\label{statement:FRC_KW}
Assume a time-net
$\tau_N=\{ 0=t_0< t_1 <\cdots < t_N =1\}$ and
$\xi = \Phi(X_{t_1},\ldots,X_{t_N})\in L_2$
with a Borel function $\Phi:\R^N\to \R$. Then one has
\[   \| H_{{\tau_N}} \xi-  P_{{\tau_N}}^X \xi   \|_{L_2}^2
    \le \sqrt{c_\eqref{eqn:statement:projection-comparison}}\, |\tau_N|^\frac{1}{2} \var(\xi)^\frac{1}{2}. \]
\end{theorem}
\medskip

It holds that for any $\xi \in L_2$ and $\tau_N=\{0=t_0<t_1<\cdots <t_N=1\}$
the projection $H_{\tau_N}$ can be written as

\begin{equation}\label{eq:proj_1st_chaos}
H_{\tau_N}\xi = \sum_{n = 1}^{N}\int_{t_{n-1}}^{t_n} \E_{t_{n-1}} k_t dW_t \mbox{ a.s.},
\end{equation}
where $k=(k_t)_{t\in [0,1]}$ is defined by the representation property \eqref{eq:representation_property} and
$(\E_{t_{n-1}} k_t)_{t\in (t_{n-1},t_n]}$ is obtained by
\[ (\E_{t_{n-1}} k_t)_{t\in (t_{n-1},t_n]} := \one_{(t_{n-1},t_n]} \E [ k| \cB([0,1])\otimes \cF_{t_{n-1}}].\]
In the next Lemma we use the notation $A=B\pm C$ for $B-C \le A \le B+C$.
We compare the Kunita Watanabe projection $P_{\tau_N}^X$ (see \eqref{eq:kunita_watanabe})
and the first running chaos $H_{\tau_N}$ (see \eqref{eq:proj_1st_chaos}).
\medskip

\begin{lemma} \label{projection-comparison}
For $\xi=\int_0^1 k_u dW_u \!\!\in L_2$ and
$\tau_N=\{0=t_0<t_1<\cdots< t_N=1\}$ one has
\begin{multline}\label{eqn:statement:projection-comparison}
     \| H_{{\tau_N}} \xi-  P_{{\tau_N}}^X \xi   \|_{L_2}^2 = \\
   \sum_{n=1}^N \int_{\tnm}^{\tn} \E
                   \bet     \E_{\tnm} k_s
                          - \frac{\E_{t_{n-1}}\int_{\tnm}^{\tn} k_udu}{\tn-\tnm}
                   \rag^2 \E \frac{\sigma(X_s)^2}{\sigma(X_{t_{n-1}})^2} ds
     \pm c_\eqref{eqn:statement:projection-comparison} \, |{\tau_N}| \var(\xi).
\end{multline}
where $c_\eqref{eqn:statement:projection-comparison}>0$ is an absolute constant, in particular not depending on $N$.
\end{lemma}
\smallskip

\begin{proof} 
(a) We start by defining additionally 
\begin{align*}
\hat{P}_{\tau_N}^X \xi & := \sum_{n=1}^N \hat{g}_{n-1}^N (X_{t_n}-X_{t_{n-1}}) \sptext{1}{with}{1}
      \hat{g}_{n-1}^N
   := \frac{\E_{t_{n-1}} \int_{t_{n-1}}^{t_n} k_t dt}
           {(t_n-t_{n-1})\s(X_{t_{n-1}})}, \\
\hat{H}_{\tau_N}\xi & := \sum_{n=1}^N \frac{\int_{t_{n-1}}^{t_n} \left[\E_{t_{n-1}} k_s\right]\sigma(X_s) dW_s}{\sigma(X_{t_{n-1}})}.
\end{align*}
Then we have
\begin{multline*}
     \|P_{\tau_N}^X \xi - H_{\tau_N}\xi\|_{L_2}^2 \\
\leq \|P_{\tau_N}^X \xi - \hat{P}_{\tau_N}^X\xi\|_{L_2}^2
      +\|\hat{P}_{\tau_N}^X\xi - \hat{H}_{\tau_N}\xi\|_{L_2}^2
      +\|\hat{H}_{\tau_N}\xi - H_{\tau_N}\xi\|_{L_2}^2.
\end{multline*}
We will estimate the terms of the right hand side.
If $\sigma\equiv 1$, then the first and last term on the RHS coincide by definition.
To estimate the first term for $\sigma(x)=x$ we let
\begin{align*}
      \al_n^N
&:=  \left \| \frac{X_{t_n}-X_{t_{n-1}}}{X_{t_{n-1}}}\right \|_{L_2}, \\
     \beta_n^N
&:=  \sup_{t\in [t_{n-1},t_n]}
     \left \| \frac{\sigma(X_t)}{\sigma(X_{t_{n-1}})} \frac{1}{e^{t_n-t_{n-1}}-1}
                     - \frac{1}{t_n-t_{n-1}} \right \|_{L_2},
\end{align*}
and get that
\begin{align*}
&    \|P_{\tau_N}^X \xi - \hat{P}_{\tau_N}^X \xi \|_{L_2}^2 \\
& =  \sum_{n=1}^N \E
      \kla   g_{n-1}^N - \hat{g}_{n-1}^N
      \mer^2 (X_{t_n}-X_{t_{n-1}})^2 \\
& =  \sum_{n=1}^N (\alpha_n^N)^2 \E \kla (g_{n-1}^N - \hat{g}_{n-1}^N) \sigma(X_{t_{n-1}}) \mer^2 \\
& =  \sum_{n=1}^N (\al_n^N)^2 \E
      \left | \E_{t_{n-1}} \int_{t_{n-1}}^{t_n}
              \left [ \frac{X_t}{X_{t_{n-1}}} \frac{1}{e^{t_n-t_{n-1}}-1}
                      - \frac{1}{t_n-t_{n-1}} \right ] 
              k_t dt 
     \right |^2 \\
&\le \sum_{n=1}^N (\al_n^N)^2 \E
      \bigg | \int_{t_{n-1}}^{t_n} \be_n^N
      (\E_{t_{n-1}} |k_t|^2)^\frac{1}{2} dt \bigg |^2 \\
&\le \sup_n (\al_n^N)^2 (\be_n^N)^2 \sum_{n=1}^N \E
      \left | \int_{t_{n-1}}^{t_n}
              (\E_{t_{n-1}} |k_t|^2)^\frac{1}{2} dt 
      \right |^2 \\
&\le \sup_n (\al_n^N)^2 (\be_n^N)^2 \sum_{n=1}^N \E
      \left | \sqrt{t_n-t_{n-1}} \kla \int_{t_{n-1}}^{t_n}
              \E_{t_{n-1}} |k_t|^2 dt \mer^\frac{1}{2} 
      \right |^2 \\
&\le \sup_n (\al_n^N)^2 (\be_n^N)^2|{\tau_N}|
      \var(\xi),
\end{align*}
where we used that $Y_t := \left [ \frac{X_t}{X_{t_{n-1}}} \frac{1}{e^{t_n-t_{n-1}}-1}- \frac{1}{t_n-t_{n-1}} \right ]$
is independent from $\mathcal{F}_{t_{n-1}}$ to get 
\begin{multline*}
        \E_{t_{n-1}}(Y_tk_t)
   \leq \left(\mathbb{E}_{t_{n-1}}Y_t^2\right)^\frac12\left(\mathbb{E}_{t_{n-1}}|k_t|^2\right)^\frac12
    =    \left(\mathbb{E}Y_t^2\right)^\frac12\left(\mathbb{E}_{t_{n-1}}|k_t|^2\right)^\frac12 \\
   \leq \beta_n^N\left(\mathbb{E}_{t_{n-1}}|k_t|^2\right)^\frac12.
\end{multline*}
Finally, by a standard computation we see that
\[ \sup_{N\ge 1} \sup_{n=1,\ldots,N}
   \al_n^N \be_n^N < \infty.\]

(b) Next we observe that 
\begin{align*}
&    \| \hat{H}_{\tau_N}\xi - \hat{P}_{\tau_N}\xi \|_{L_2}^2 \\
& =  \mathbb{E}\left(\sum_{n=1}^N \left(\int_{t_{n-1}}^{t_n}\frac{\left[\mathbb{E}_{t_{n-1}} k_s\right]}{\sigma(X_{t_{n-1}})}\sigma(X_s)dW_s - \int_{t_{n-1}}^{t_n}\hat{g}^N_{n-1}\sigma(X_s)dW_s\right)\right)^2  \\
& =  \sum_{n=1}^N \int_{\tnm}^{\tn} \E
        \bet   [ \E_{\tnm} k_s]
                \frac{1}{\sigma(X_{t_{n-1}})}
        - \hat{g}_{n-1}^N \rag^2 \sigma(X_s)^2 ds \\
& =  \sum_{n=1}^N \int_{\tnm}^{\tn} \E
        \bet   [ \E_{\tnm} k_s]
                \frac{1}{\sigma(X_{t_{n-1}})}
        - \frac{\int_{\tnm}^{\tn} \E_{t_{n-1}}k_udu}
               {(\tn-\tnm)\s(X_{t_{n-1}})}
       \rag^2 \sigma(X_s)^2 ds \\
& =  \sum_{n=1}^N \int_{\tnm}^{\tn} \E
        \bet   [ \E_{\tnm} k_s]
               - \frac{\int_{\tnm}^{\tn} \E_{t_{n-1}}k_udu}
               {\tn-\tnm}
       \rag^2 \frac{\sigma(X_s)^2}{\sigma(X_{t_{n-1}})^2} ds \\
& =  \sum_{n=1}^N \int_{\tnm}^{\tn} \E
        \bet   [ \E_{\tnm} k_s]
               - \frac{\int_{\tnm}^{\tn} \E_{t_{n-1}}k_udu}
               {\tn-\tnm}
       \rag^2 \E \frac{\sigma(X_s)^2}{\sigma(X_{t_{n-1}})^2} ds.
\end{align*}

(c) Finally, again we assume that $\sigma(x)=x$ and get that
\equa
      \| \hat{H}_{\tau_N}Z - H_{\tau_N}Z \|_{L_2}^2
& = & 
      \sum_{n=1}^N \E \int_{t_{n-1}}^{t_n} 
      \bet \frac{[\E_{t_{n-1}} k_s]}{X_{t_{n-1}}} X_s -
           [\E_{t_{n-1}} k_s] \rag^2 ds \\
& = & \sum_{n=1}^N \int_{t_{n-1}}^{t_n} 
      \E \bet [\E_{t_{n-1}} k_s] \rag^2 
      \E \bet \frac{X_s}{X_{t_{n-1}}} - 1 \rag^2 ds \\
& = & \sum_{n=1}^N \int_{t_{n-1}}^{t_n} 
      \E \bet [\E_{t_{n-1}} k_s] \rag^2 
      (e^{s-t_{n-1}}-1) ds \\
&\le& e^{|\tau_N|} |{\tau_N}| \var(\xi),
\tion
and the statement follows. For the last inequality we use the mean value theorem to get the inequality
$e^{s - t_{n-1}} -1 = (s - t_{n-1})e^{s'} \leq |\tau_N|e^{|\tau_N|}$ for some $s' \in [0, s-t_{n-1}]$.
\end{proof}
\medskip

\begin{lemma}\label{statement:integrand_piecewise_martinagele}
For $0=t_0<t_1<\cdots< t_N=1$ and
$\xi=\Phi(X_{t_1},\ldots,X_{t_N})\in L_2$ with a Borel function $\Phi:\R^N\to \R$
there is a representation
\[ \xi = \E \xi + \int_0^1 k_u dW_u \mbox{ a.s.} \]
such that $(k_u)_{u\in [t_{n-1},t_n)} \subseteq L_2$  are continuous  $(\cF_u)_{u\in [t_{n-1},t_n)}$-martingales.
\end{lemma}
\medskip

\begin{proof}
We consider the chaos expansion of  $Z=\Phi(X_{t_1},\ldots,X_{t_N})\in L_2$ given by
\[  \xi = \E \xi +  \sum_{n=1}^\infty I_n (h_n) \mbox{ a.s.}  \]
such that the $(h_n)_{n\ge 1}$ are symmetric measurable kernels
$h_n:(0,1]^n \to \R$ such that $\int_0^1 \cdots \int_0^1 |h_n(t_1,\ldots,t_n)|^2 d t_n \cdots d t_1 < \infty$.
By \cite[Proposition 6.5]{Baum-Geiss}   each  function $h_n$  is constant on the cuboids
\[  (t_{k_1},t_{k_1+1}] \times ... \times(t_{k_n},t_{k_n+1}]  \quad \text{for} \,\, t_{k_j}  \in \{1,...,N\}. \]
Since
\[ I_n(h_n) := n! \int_0^1 \int_0^{t_n} \dots \int_0^{t_2} h_n(t_1, \dots, t_n) dW_{t_1} \dots dW_{t_n}.
   \]
we have
\[ \sum_{n=1}^\infty I_n (h_n)  = \int_0^1  \sum_{n=1}^\infty n I_{n-1} ( h_n (u, \cdot)) dW_u \mbox{ a.s.} \]
Then
$$k_u =  \sum_{n=1}^\infty n I_{n-1} ( h_n (u, \cdot))   $$
is constant on $(t_{k}, t_{k +1}] $  for $k =0,...,n-1.$
\end{proof}
\smallskip

\begin{proof}[Proof of Theorem \ref{statement:FRC_KW}]
The statement follows from \cref{projection-comparison} and
 \cref{statement:integrand_piecewise_martinagele}.
\end{proof}
\bigskip

\begin{lemma}\label{lemma:Z-Z_N_variance_new}
Assume that $Z$ has the admissible representation \eqref{eqn:definition:admissible_representation_Z}. 
Let $(\tau_N)_{N>M}$ be a sequence of times nets $0=t_{0,N} < t_{1,N} < \cdots < t_{N,N}=1$
such that $\{s_1,\ldots,s_M\} \subseteq \tau_N$ and  $\lim_{N\to \infty} |\tau_N| =0$.
Then
\[ \lim_{N\to \infty} \|Z-Z_{\tau_N}\|_{L_2} = 0. \]
\end{lemma}

\begin{proof}
To shorten the notation we will write
\begin{equation} \label{short-phi}
 \vph(t):=\vph(t, X_t, L)
 \sptext{1}{and use}{1}
  (\delta_N \Psi) (u):= \Psi(u) - \Psi(l_N(u))
\end{equation}
for a function $\Psi: [0,1] \to \R$, where
\begin{equation} \label{interval}
l_N(t)          := \sum_{n=1}^N \one_{(t_{n-1,N},\,t_{n,N}]}(t) t_{n-1,N}.
\end{equation}
In our setting it holds according to \cref{prop:skorohod_IBP} that, a.s.,
\begin{align*}
    Z_{\tau_N}
& = \sum_{n=1}^N  \vph(t_{n-1,N}, X_{t_{n-1,N}}, L) (X_{t_{n,N}}- X_{t_{n-1,N}}) \\
& = \int_0^1  \vph(l_N(t),X_{l_N(t)}, L)\sigma(X_t)  \de W_t \\
&\quad  + \int_0^1 \sigma(X_t) \nabla_y\vph(l_N(t),X_{l_N(t)}, L) \cdot D_t L \, dt.
\end{align*}
Consequently, a.s.,
\begin{align} \label{Z-Z-N}
      Z-Z_{\tau_N}
& =   \int_0^1  (\delta_N\vph)(t) \sigma(X_t)  \de W_t
     + \int_0^1 \sigma(X_t)
          (\delta_N \nabla_y\vph)(t) \cdot D_t L dt  \notag \\
& =   \int_0^T  (\delta_N\vph)(t) \sigma(X_t)  \de W_t
     + \int_0^T \sigma(X_t)
          (\delta_N \nabla_y\vph)(t) \cdot D_t L dt    \\
& +  \int_T^1  (\delta_N\vph)(t) \sigma(X_t)  d W_t. \notag
\end{align}
In case of $T<1$ for the last term we know by
\cite[Theorem 3.2 and its proof]{Gei19} that
\[ \lim_{N\to \infty} \E \left [  \int_T^1  (\delta_N\vph)(t) \sigma(X_t)  d W_t \right ]^2 =0.\]
For the first term we use \cref{prop:skorohod_covariance} and get that
\begin{align*}
&      \E \left |
      \int_0^T
      (\delta_N\vph)(t) \sigma(X_t)  \de W_t \right |^2 \\
= & \E \int_0^T \bet
      \bS(X_t)(\delta_N \vph)(t) \rag^2 dt \\
&      + 2 \E \int_0^T \int_0^t
                           \left [ D_t \Big ( \bS(X_s)(\delta_N\vph)(s) \Big ) \right ]
                           \left [ D_s \Big ( \bS(X_t)(\delta_N\vph)(t) \Big ) \right ]
                        ds dt \\
= & \E \int_0^T \bet \bS(X_t)(\delta_N \vph)(t) \rag^2 dt
   + 2 \E \int_0^T \int_0^t \bS(X_s) \left [ \nabla_y (\delta_N\vph)(s) \cdot D_t L\right ]  \\
& \! \times \!\Big [( D_s  \bS(X_t))(\delta_N\vph)(t) \\
& \quad + \bS(X_t)
    [(\delta_N\vph_x)(t) D_s X_t + \vph_x(l_N(t)) (D_s X_t - D_s X_{l_N(t)})] \\
& \quad + \bS(X_t) (\nabla_y (\delta_N\vph) (t) \cdot D_s L) \Big ] ds dt.
\end{align*}
Because of the regularity properties of $\vph$ on all intervals $(s_{m-1},s_m)$
with $s_m\le T$  we may use dominated convergence to show convergence to zero as $|\tau_N| \to 0$.
Moreover, for the second term in \eqref{Z-Z-N} we have
\[ \E \bet \int_0^T \sigma(X_t)(\delta_N \nabla_y\vph)(t)\cdot D_t Ldt \rag^2
   \to_N 0 \]
as well, which can be shown via a similar argument.
\end{proof}


\section{Proof of main results}
\label{sec:proof_main_results}


\subsection{A technical lemma}

The following lemma uses ideas from \cite[Proof of Lemma 5.3]{Geiss:Toivola:2009}:

\begin{lemma}\label{lemma:convergence_to_diagonal}
Assume $\theta \in (0,1]$ and for $N>M$ Borel measurable $A_N:[0,1]\times[0,1] \to [-C,C]$ for some $C>0$
and an $A:[0,1] \to \R$  such that for the nets $\sigma_N^{\theta}$ from \cref{def:time_nets} the following conditions
are satisfied:
\begin{enumerate}[{\rm (1)}]
\item  For a.a. $s\in (0,1)$  it holds
       $\lim\limits_{N\to\infty} \sup\limits_{l_N(s) < t<s} |A^N(t,s) -A^N(s,s)| =0$.
\item $\lim_{N \to \infty} A^N(s,s) = A(s)$ for a.a. $s\in (0,1),$
\item \label{item:3:lemma:convergence_to_diagonal}
      $A$ is left-continuous with right limits and has at most a finite number of discontinuities.
\end{enumerate}
Then one has that
\[   \lim_{N\to \infty} N \int_0^1\int_{l_N(s)}^s A^N(t,s) dtds
   = \frac{1}{2\theta} \int_0^1 (1-s)^{1-\theta} A(s)  ds. \]
\end{lemma}

\begin{proof}
(a) Denote by $\Xi_\theta(t):[0,1] \to [0,1]$ the map $\Xi_\theta(t) = 1 - (1-t)^\frac{1}{\theta}$ so that 
$\Xi_\theta(\frac{i}{R}) = t_{i, R}^\theta$.
For $t\in [0,1]$ and $R\in \N$ consider the distribution functions
\[ \mu_R^\theta(t) := \sum_{i = 1}^R R\int_{t_{i-1, R}^\theta \wedge t}^{t_{i, N}^\theta \wedge t}\left(s -t_{i-1, R}^\theta\right)ds
   \sptext{1}{and}{1} 
   \mu^\theta(t) := \frac{1}{2\theta}\int_0^t (1-s)^{1-\theta} ds.\]
For $t\in [0,1]$ and $r_{i,R}^\theta:= t_{i,R}^\theta \wedge t$ we get 
some
$\tilde r_{i,R}^\theta \in  \left [ r_{i-1,R}^\theta,r_{i,R}^\theta \right ]$ 
such that 
\begin{align*}\label{eq:mu_N_theta_1}
    \mu_R^\theta(t) 
& = \frac{R}{2} \sum_{i = 1}^R \left(r_{i,R}^\theta - r_{i-1,R}^\theta \right)^2 \\
& = \frac{R}{2} \sum_{i = 1}^R \left(r_{i,R}^\theta - r_{i-1,R}^\theta \right) 
    \left[ \Xi_\theta\left(\Xi_\theta^{-1}(r_{i,R}^\theta) \right ) - \Xi_\theta\left(\Xi_\theta^{-1}(r_{i-1,R}^\theta) \right ) \right]\\
& = \frac{R}{2} \sum_{i = 1}^R \left(r_{i,R}^\theta - r_{i-1,R}^\theta \right) 
    \Xi_\theta'( \Xi_\theta^{-1}(\tilde r_{i,R}^\theta))  \left [ \Xi_\theta^{-1}(r_{i,R}^\theta) - \Xi_\theta^{-1}(r_{i-1,R}^\theta) \right ] \\
& = \frac{1}{2\theta} \sum_{i = 1}^R \left(r_{i,R}^\theta - r_{i-1,R}^\theta \right) 
    (1-\tilde r_{i,R}^\theta)^{1-\theta} \left [ R  \left [ \Xi_\theta^{-1}(r_{i,R}^\theta) - \Xi_\theta^{-1}(r_{i-1,R}^\theta)
    \right ] \right ],
\end{align*}
where we use the mean value theorem to determine $\tilde r_{i,R}^\theta$. 
Now we observe that for $t\le t_{i-1,R}^\theta$ one has
$\Xi_\theta^{-1}(r_{i,R}^\theta) - \Xi_\theta^{-1}(r_{i-1,R}^\theta)=0$ and for 
$t\ge t_{i,R}^\theta$ one has
$N  \left [ \Xi_\theta^{-1}(r_{i,R}^\theta) - \Xi_\theta^{-1}(r_{i-1,R}^\theta)\right ] =1$. This implies that 
\[\lim_{R\to \infty} \mu_R^\theta(t) = \frac{1}{2\theta}\int_0^t (1-s)^{1-\theta} ds = \mu^\theta(t). \]
\medskip

(b) Denote by $\nu_N^\theta$ the Borel measure on $\cB([0,1])$ with the insider time points included, i.e.
\[ d\nu_N^\theta(s) := N(s-l_N(s))\,ds. \]
Because  
\[ 0\le \mu_R^\theta([0,1]) - \frac{R}{R+M} \nu_{R+M}^\theta([0,1])
           \le R M \frac{|\tau_R^\theta|^2}{2} \to 0,
      \]
we get for $t\in [0,1]$ that
\[\lim_{N\to \infty} \nu_N^\theta(t) = \frac{1}{2\theta}\int_0^t (1-s)^{1-\theta} ds = \mu^\theta(t). \]
(c) As $\nu_N^\theta(\{t\}) = \mu^\theta(\{t\})=0$ for all $t\in [0,1]$ by Item
\eqref{item:3:lemma:convergence_to_diagonal} it follows that 
\[ \lim_{N\to \infty} \int_0^1 A(s) \nu_N^\theta(ds) =\int_0^1 A(s)\,\mu^\theta(ds). \]
We observe that we can replace the nets $\tau_R^\theta$ by $\sigma_N^\theta$ in (a):
\medskip

(d) Now the statement follows from
\begin{align*}     
&     \bigg | N \int_0^1\int_{l_N(s)}^s A^N(t,s) dtds - \int_0^1 N(s-l_N(s)) A(s) ds \bigg | \\
&\le  \bigg | \int_0^1N \int_{l_N(s)}^s (A^N(t,s) -A^N(s,s) )dtds \bigg | \\
&\quad  + \bigg |  \int_0^1  N (s-l_N(s)) (A^N(s,s) -A(s)) ds \bigg | \\
&\le   \bigg | \int_0^1 N (s-l_N(s)) \sup_{l_N(s) < t<s} |A^N(t,s) -A^N(s,s)| ds \bigg |  \\
&\quad  + \bigg |  \int_0^1  N (s-l_N(s)) (A^N(s,s) -A(s)) ds \bigg | \to 0,
\end{align*}
where the convergence towards zero follows from dominated convergence and
 $0\le N(s-l_N(s))\le N|\sigma_N^\theta|\le c_\theta$.
\end{proof}


\subsection{Projection onto the first running chaos}

In this section, in \cref{sec:2nd_chaos}, and in \cref{sec:proof:thm:error_limit}
we follow the convention that integrals over time are understood as to be taken over the indicated
range, but without grid points from $\sigma_N^\theta$. This does not affect the integrals, but avoids
unnecessary cases regarding the function $l_N(u)$, namely that at a grid point $l_N$ returns
the previous grid point.
\medskip

For the following we let
\begin{align*}
Z^T & := \sum_{m=1}^M \left [ F_m(s_m,X_{s_m},L) -  F_m(s_{m-1},X_{s_{m-1}},L)\right ],\\
Z_{\tau}^T & := \sum_{t_n \le T}  \vph(t_{n-1},X_{t_n-1},L)
               (X_{t_n}-X_{t_{n-1}})
\end{align*}
with $T\in \tau = \{0 = t_0 < t_1 < \dots < t_N = 1\}$.
\bigskip

\begin{proposition} \label{first-chaos} Let $\sigma_N^\theta$, $N>M$, be the nets from \cref{def:time_nets}.
For any admissible representation $Z$ it holds that
\begin{multline*}
\lim_{N\to \infty} N \| H_{\sigma_N^{\theta}}(Z^T - Z^T_{\sigma_N^\theta} )\|_{L_2}^2 \\
=  \frac{1}{2\theta} \int_0^T (1-t)^{1-\theta}
      \int_t^T \| \E_s [ \sigma^2(X_t)\nabla_y \vph_x(t,X_t,L) \cdot D_s L]\|_{L_2}^2
      dsdt.
\end{multline*}
\end{proposition}
\bigskip

To prepare the proof we first notice that
$Z^T,Z^T_{\sigma_N^\theta}\in \DD_{1,2}$ and that we have
\[ H_{\sigma_N^\theta}(Z^T - Z^T_{\sigma_N^\theta} ) 
   = \int_0^T \E_{l_N(s)} D_s (Z^T- Z^T_{\sigma_N^\theta} ) dW_s \mbox{ a.s.},
\]
where $l_N(s)$ is defined in \eqref{interval}. In the following Lemma we determine the structure of   
$\E_{l_N(s)} D_s (Z^T- Z^T_{\sigma_N^\theta} )$,
where we use \eqref{Z-Z-N} to write
\begin{align}\label{D-Z-Z-N}
& (Z^T- Z^T_{\sigma_N^\theta} ) \notag  \\
& = \int_0^T  (\delta_N\vph)(t) \sigma(X_t)  \de W_t
     + \int_0^T \sigma(X_t) (\delta_N \nabla_y\vph)(t) \cdot D_t L dt \mbox{ a.s.}
\end{align}
\medskip

\begin{lemma} \label{decomposition}
Let \cref{assumptionA}  hold. For a.a. $s\in (0,T)\setminus \sigma_N^\theta$ we have that
\begin{align}\label{eqn:1:decomposition}
& \E_{l_N(s)} D_s \int_0^T (\delta_N\vph)(u) \sigma(X_u)  \de W_u  \notag \\
&=  \int_0^{l_N(s)}  \int_{l_N(u)}^u
                  \eta(u,t;s) \delta W_t \delta W_u
                  + \int_{l_N(s)}^s \nu(t;s) dt \\
& + \int_0^{l_N(s)} \int_{l_N(u)}^u
                  \lambda(u,t;s) dt \delta W_u \mbox{ a.s.}, \notag
\end{align}
where, for $l_N(u) < t < u < l_N(s)$ with $u\not\in \sigma_N^\theta$ and $l_N(s) < t < s < u_N(s)$,
\begin{align*}
    \eta(u,t;s)
& := \E_{l_N(s)} \Big[ \sigma(X_u)\sigma(X_t) \nabla_y \vph_x(t,X_t,L)
    \cdot D_s L \Big], \\
    \nu(t;s)
& := \E_{l_N(s)} \bigg[ \vph_x(t,X_t,L) \sigma(X_t) D_t\bS(X_s) \\
& \quad + \bS(X_s)\sigma(X_t)\nabla_y\vph_x(t,X_t,L) \cdot D_t L \\
& \quad  - \bS(X_s)\sigma(X_t)\sigma'(X_t) \vph_x(t,X_t,L) \bigg], \\
    \lambda(u,t;s)
& := \E_{l_N(s)} D_s \left[\sigma(X_u) \sigma(X_t)\nabla_y \vph_x(t,X_t,L) \cdot D_t L\right] \\
& \quad - \E_{l_N(s)} D_s \left[\sigma(X_u) \sigma(X_t)\sigma'(X_t) \vph_x(t,X_t,L)\right]\\
& \quad + \E_{l_N(s)} \Big[ (D_t \sigma(X_u)) D_s \vph_x(t,X_t,L) \sigma(X_t)\Big]
\end{align*}
and
\begin{align} \label{eqn:2:decomposition}
& \E_{l_N(s)} D_s \int_0^T  \sigma(X_t) (\delta_N\nabla_y\vph)(t)
  \cdot D_t L dt \notag \\
& =  \E_{l_N(s)}  \int_0^T \int_{l_N(t)}^t
    \delta(t,r;s) dr dt +  \E_{l_N(s)} \int_s^{u_N(s)}  \gamma(t;s) \, dt \\
& \quad +  \int_0^T  \int_{l_N(t)}^{t\wedge l_N(s)} \E_{l_N(s)}  D_s \gamma(t;r) \delta W_r dt,   \notag
\end{align}
where, for $l_N(t)<r<t$, $s<t<u_N(s)$, and $l_N(t)<r<t\wedge l_N(s)$ with $t\not \in \sigma_N^\theta$,
\begin{align*}
    \delta(t,r;s)
&:= D_s \bigg( \nabla_y \vph_x(r,X_r,L) \sigma(X_r)  D_r(D_t L \sigma(X_t)) \\
&   \quad + \sigma(X_t)\sigma(X_r) \nabla_y\nabla_y\vph_x(r,X_r,L)\cdot D_r L \cdot D_t L \\
&   \quad -\sigma(X_t) \sigma(X_r)\sigma'(X_r)\nabla_y  \vph_x(r,X_r,L) \cdot D_t L \bigg) \\
    \gamma(t;r)
&:= \sigma(X_t) \nabla_y \vph_x(r,X_r,L) \sigma(X_r)\cdot D_t L.
\end{align*}
\end{lemma}
\bigskip

\begin{proof}
\eqref{eqn:1:decomposition}
By \cref{prop:derivative_of_skorohod_integral} and \cref{lemma:expectation_of_skorohod} we have for a.a. 
$s\in (0,T) \setminus \sigma_N^\theta$ that
\begin{align}\label{first-decomposition}
&   \E_{l_N(s)} D_s \bigg( \int_0^T (\delta_N\vph)(u) \sigma(X_u)  \de W_u  \bigg)  \notag \\
& = \E_{l_N(s)} (\delta_N\vph)(s) \bS(X_s) + \int_0^{l_N(s)} \E_{l_N(s)} D_s\left[(\delta_N\vph)(u)
    \bS(X_u)\right] \delta W_u \mbox{ a.s.}
\end{align}
By It\^o's formula (\cref{lemma:ito}) we have for $u\not \in \sigma_N^\theta$ that, a.s.,
\begin{align} \label{ito}
     \delta_N \vph(u) 
& =  \vph(u,X_u,L)-\vph(l_N(u),X_{l_N(u)},L) \notag \\
& =  \int_{l_N(u)}^u \vph_t(t,X_t,L) dt + \int_{l_N(u)}^u \vph_x(t,X_t,L) \sigma(X_t) \delta W_t \notag \\
& \quad +  \int_{l_N(u)}^u  \sigma(X_t)\nabla_y \vph_{x}(t,X_t,L)\cdot D_tL dt \notag \\
& \quad + \frac{1}{2} \int_{l_N(u)}^u \vph_{xx}(t,X_t,L) \sigma^2(X_t) dt \notag \\
& =  \int_{l_N(u)}^u \vph_x(t,X_t,L) \sigma(X_t) \delta W_t \\
& \quad +  \int_{l_N(u)}^u  \sigma(X_t)\nabla_y \vph_{x}(t,X_t,L)\cdot D_tL dt \notag \\
& \quad - \int_{l_N(u)}^{u}\sigma(X_t)\sigma'(X_t) \vph_x(t,X_t,L) dt. \notag
\end{align}

\medskip\noindent
\textbf{First term of \eqref{first-decomposition}:}
Applying \eqref{ito} and \cref{prop:skorohod_IBP} yields
\begin{align*}
    \E_{l_N(s)} \bS(X_s)(\delta_N\vph)(s)
& = \E_{l_N(s)} \int_{l_N(s)}^s \bS(X_s) \vph_x(t,X_t,L) \sigma(X_t) \delta W_t\\
&\quad  + \int_{l_N(s)}^s \E_{l_N(s)} \bigg[
      \vph_x(t,X_t,L)\sigma(X_t) D_t\bS(X_s)\\
&\quad+ \bS(X_s)\sigma(X_t) \nabla_y\vph_x(t,X_t,L) \cdot D_t L \\
&\quad- \bS(X_s)\sigma(X_t)\sigma'(X_t) \vph_x(t,X_t,L)\bigg] dt,
\end{align*}
and using  \cref{lemma:expectation_of_skorohod}
to notice that the
conditional expectation of the
Skorohod integral vanishes gives  $\nu(t;s)$.
\medskip

\textbf{Second term of \eqref{first-decomposition}:}
Substituting \eqref{ito} into the second term of
\eqref{first-decomposition} yields 
\begin{align*}
&   \int_0^{l_N(s)} \E_{l_N(s)} D_s\big[(\delta_N\vph)(u) \bS(X_u)\big] \delta W_u \\
& = \int_0^{l_N(s)} \E_{l_N(s)} D_s \bigg[
      \int_{l_N(u)}^u \vph_x(t,X_t,L)
      \sigma(X_t) \delta W_t \,  \sigma(X_u)
      \bigg] \delta W_u \\
& \quad + \int_0^{l_N(s)} \int_{l_N(u)}^u \E_{l_N(s)} D_s \bigg( \Big[
      \sigma(X_t)\nabla_y \vph_x(t,X_t,L)
      \cdot D_t L \\
& \quad \hspace*{11em}
      -\sigma(X_t)\sigma'(X_t)\vph_x(t,X_t,L)\Big] \sigma(X_u) \bigg) dt \delta W_u.
\end{align*}
The integrand of the second term are the first two terms of  $\lambda(u,t;s)$.
Using \cref{prop:skorohod_IBP} and noting $t < u \le l_N(s) < s$ we get
\begin{align*}
  & \E_{l_N(s)} D_s \bigg[
      \int_{l_N(u)}^u \vph_x(t,X_t,L)
      \sigma(X_t) \delta W_t
      \sigma(X_u)\bigg] \\
& = \int_{l_N(u)}^u \E_{l_N(s)} \Big[
      \sigma(X_u)\sigma(X_t)  D_s \vph_x(t,X_t,L)
      \Big] \delta W_t \\
& \quad  + \int_{l_N(u)}^u \E_{l_N(s)} \Big[
      (D_t \sigma(X_u)) D_s \vph_x(t,X_t,L) \sigma(X_t)
      \Big] dt.
\end{align*}
This yields $\eta(u,t;s)$ and the third term of $\lambda(u,t;s)$.
\medskip

\eqref{eqn:1:decomposition} For $t\not \in \sigma_N^\theta$ we get
by It\^o's formula (\cref{lemma:ito}),
\begin{align*}
     (\delta_N\nabla_y\vph)(t)
& =  \int_{l_N(t)}^t \nabla_y \vph_x(r,X_r,L) \sigma(X_r) \delta W_r \\
&   \quad + \int_{l_N(t)}^t  \sigma(X_r)\nabla_y\nabla_y\vph_x(r,X_r,L) \cdot D_r L  dr\\
&   \quad - \int_{l_N(t)}^t\sigma(X_r)\sigma'(X_r)\nabla_y  \vph_x(r,X_r,L) dr \mbox{ a.s.},
\end{align*}
so that by  \cref{prop:skorohod_IBP},
\begin{align*}
&   \hspace{-3em} \sigma(X_t)(\delta_N\nabla_y\vph)(t) \cdot D_t L\\
& = \int_{l_N(t)}^t \sigma(X_t) \nabla_y \vph_x(r,X_r,L) \sigma(X_r)\cdot D_t L \, \delta W_r \\
& \quad + \int_{l_N(t)}^t\nabla_y \vph_x(r,X_r,L) \sigma(X_r)\cdot     D_r ( \sigma(X_t) \, D_t L) dr \\
& \quad + \int_{l_N(t)}^t\sigma(X_t) \sigma(X_r)\nabla_y^2\vph_x(r,X_r,L) \cdot (D_r L\otimes D_t L) dr \\
& \quad - \int_{l_N(t)}^t  \sigma(X_t)\sigma(X_r)\sigma'(X_r)\nabla_y\vph_x(r,X_r,L) \cdot D_t L dr \mbox{ a.s.}
\end{align*}
Using that $\one_{(l_N(t), t]}(s) = \one_{[s, u_N(s))}(t)$ we get
\begin{align*}
&   \E_{l_N(s)} D_s \int_0^T \sigma(X_t) (\delta_N\nabla_y\vph)(t) \cdot D_t L\, dt\\ 
& = \mathbb{E}_{l_N(s)}  \int_0^T [\sigma(X_t) \nabla_y \vph_x(s,X_s,L) \sigma(X_s)\cdot D_t L  ]\one_{(l_N(t), t]}(s) \, dt  \\
&   \quad +  \mathbb{E}_{l_N(s)} \int_0^T  \int_{l_N(t)}^t D_s (\sigma(X_t) \nabla_y \vph_x(r,X_r,L) \sigma(X_r)\cdot
     D_t L )\delta W_r dt \\
& \quad +     \mathbb{E}_{l_N(s)}  \int_0^T \int_{l_N(t)}^t \delta(t,r;s) dr dt\\
& = \E_{l_N(s)} \int_s^{u_N(s)} \sigma(X_t) \nabla_y \vph_x(s,X_s,L) \sigma(X_s)\cdot D_t L   \, dt  \\
& \quad + \mathbb{E}_{l_N(s)} \int_0^T  \int_{l_N(t)}^t D_s (\sigma(X_t) \nabla_y \vph_x(r,X_r,L) \sigma(X_r)\cdot
          D_t L )\delta W_r dt \\
& \quad +     \mathbb{E}_{l_N(s)}  \int_0^T \int_{l_N(t)}^t \delta(t,r;s) dr dt. \qedhere
\end{align*}
\end{proof}

\begin{proof}[Proof of \cref{first-chaos}]
Since
\[ H_{\sigma_N^\theta}(Z^T - Z^T_{\sigma_N^\theta} ) = \int_0^T \E_{l_N(s)} D_s (Z^T- Z^T_{\sigma_N^\theta} ) dW_s,
\]
we have 
\[ \| H_{\sigma_N^{\theta}}(Z^T - Z^T_{\sigma_N^\theta} )\|_{L_2}^2
   =  \int_0^T \E (\E_{l_N(s)} D_s (Z- Z_{\sigma_N^\theta} ) )^2ds.  \]
Only the term
      $\int_0^{l_N(s)} \int_{l_N(u)}^u\eta(u,t;s) \delta W_t \delta W_u$
      from  \cref{decomposition} contributes to the limit expression   of 
      $$ \lim_{N\to \infty } N  \int_0^T \E (\E_{l_N(s)} D_s (Z- Z_{\tau_N^\theta} ) )^2ds.$$
It holds
\begin{align}\label{for-the-limitnew}
& \int_0^T \E \bigg | \int_0^{l_N(s)} \int_{l_N(u)}^u\eta(u,t;s) \delta W_t \delta W_u \bigg |^2ds \notag \\
&= \int_0^T \int_0^{l_N(s)} \int_{l_N(u)}^u\E (\eta(u,t;s))^2 dtduds \\
& \quad +  \E \int_0^T\int_0^{l_N(s)}\int_0^{l_N(s)} D_r  \bigg( \int_{l_N(u)}^u\eta(u,t;s) \delta W_t \bigg) \notag \\
& \hspace{7em} \times      D_u \bigg ( \int_{l_N(r)}^r\eta(r,t;s) \delta W_t \bigg)  du dr ds \notag \\
& \quad + \E \int_0^T \int_0^{l_N(s)} \int_{l_N(u)}^u\int_{l_N(u)}^u  D_r\eta (u,t;s)
    D_t\eta(u,r;s) dtdrduds, \notag
\end{align}
where we use the covariance formula for Skorohod integrals, \cref{prop:skorohod_covariance}.
We proceed with the first term of  \eqref{for-the-limitnew} and get
\begin{align*}
&  \int_0^T \int_0^{l_N(s)} \int_{l_N(u)}^u\E (\eta(u,t;s))^2 dtduds \\
&= \int_0^T \int_{l_N(u)}^u \int_{u_N(u)}^T  \E(\eta(u,t;s))^2  dsdtdu \\
&= \int_0^T \int_{l_N(u)}^u \int_{u_N(u)}^T   \E (
   \E_{l_N(s)} \left [ \sigma(X_u)\sigma(X_t)
   \nabla_y \vph_x(t,X_t,L)\cdot D_sL ) \right ])^2 dsdtdu,
\end{align*}
where
\[u_N(t) :=  \sum_{n=1}^N \one_{(\tnm,\tn]}(t) \tn. \]
Choose now
\[      A^N(t,s):=\begin{cases} \int_{u_N(s)}^T   \E (
      \E_{l_N(r)} \left [ \sigma(X_s)\sigma(X_t)
            \nabla_y \vph_x(t,X_t,L)\cdot D_rL ) \right ])^2 dr & \\
        & \hspace{-6em} t,s \in (0, T) \\
      0 & \hspace{-6em} \text{otherwise}
   \end{cases}. \]
and
\[ A(s):=\begin{cases} \int_s^T \E (\E_r [ \sigma^2(X_s)\nabla_y \vph_x(s,X_s,L) \cdot D_r L])^2 dr,
    \quad  s \in (0, T),\\
     0, \quad \text{otherwise}.   \end{cases}. \]
Then the conditions of \cref{lemma:convergence_to_diagonal} are satisfied and we have
\begin{multline*}
\hspace*{-1em} \lim_{N \to \infty} N \! \int_0^T \int_{l_N(s)}^s \! \int_{u_N(s)}^T   \!\!\!\E (
      \E_{l_N(r)} \left [ \sigma(X_s)\sigma(X_t)
            \nabla_y \vph_x(t,X_t,L)\cdot D_rL ) \right ])^2 drdtds \\
=\frac{1}{2\theta} \int_0^T (1-s)^{1-\theta}
      \int_s^T \E \bet \E_r [ \sigma(X_s)^2\nabla_y \vph_x(s,X_s,L) \cdot D_r L]\rag^2
      drds.
\end{multline*}
In the same way one can show that the second moments of the other terms
are of the order $o(N^{-1})$. This proves the assertion.
\end{proof}


\subsection{Strategy for the proof of the main theorems} \ \medskip
\label{sec:proof_strategy}

{\bf (a)} By \cref{statement:FRC_KW} and \cref{lemma:Z-Z_N_variance_new} the statements, \cref{statement:R-approximation_with_KW-projection} and \cref{thm:2nd_chaos}, can be directly derived from each other. Therefore it is
sufficient to verify \cref{thm:2nd_chaos} only.
\smallskip

{\bf (b)} The case $T<1$ under the assumption
\begin{equation}\label{eqn:H-theta-condtion}
\int_0^1 (1-t)^{1-\theta} \E \left |\sigma^2 (X_t) \frac{\partial^2 F}{\partial x^2}(t,X_t)\right |^2 dt < \infty:
\end{equation}
As $L$ is $\cF_T$-measurable, by orthogonality we have  
\begin{align*}
&    \var\left ([Z - Z_{\sigma_N^\theta}] - P_{\sigma_N^\theta}\left[Z - Z_{\sigma_N^\theta}\right] \right ) \\
& =  \var\left ([Z^T - Z^T_{\sigma_N^\theta}] - P^X_{\sigma_N^\theta}[Z^T - Z^T_{\sigma_N^\theta}] \right )  \\
&\quad  + \var\left (\left [ (Z-Z^T) - (Z_{\sigma_N^\theta}-Z^T_{\sigma_N^\theta})\right ]
    - P^X_{\sigma_N^\theta} \left [ (Z-Z^T) - (Z_{\sigma_N^\theta}-Z_{\sigma_N^\theta}^T)\right ] \right ).
\end{align*}
Then one has
\begin{align*}
& \hspace*{-2em}  \lim_{N\to \infty} N \var\Big (\left [ (Z-Z^T) - (Z_{\sigma_N^\theta}-Z^T_{\sigma_N^\theta})\right ] \\
& \hspace*{4em} - P^X_{\sigma_N^\theta} \left [ (Z-Z^T) - (Z_{\sigma_N^\theta}-Z_{\sigma_N^\theta}^T)\right ] \Big ) \\
&  = \lim_{N\to \infty} N \var([Z-Z^T]- [Z_{\sigma_N^\theta}-Z^T_{\sigma_N^\theta}]) \\
&  = \frac{1}{2\theta} \int_T^1 (1-t)^{1-\theta} H^2(t)dt
\end{align*}
with
\[ H(t) := \left \| \sigma^2(X_t)\vph_x(t,X_t)  \right \|_{L_2}.\]

The first equality follows, with the obvious modification that we are on the time interval $[T,1]$,
from \cite[Theorem 4.4]{Geiss:2002} as the first limit refers to the {\em optimal approximation} and the second one to the
{\em simple approximation} (both are treated in \cite{Geiss:2002}). To see the second equality we proceed as follows: Again using 
\cite[Theorem 4.4]{Geiss:2002} we consider
\[ \int_T^{t_{n(T),N}^\theta} (t_{n(T),N}^\theta-u) H^2(u) du + \sum_{n=n(T)+1}^N \int_{t_{n-1,N}^\theta}^{t_{n,N}^\theta} (t_{n,N}^\theta-u) H^2(u) du, \]
where $t_{n(T),N}^\theta$ is the smallest grid point larger than $T$.
Similar to \cref{lemma:convergence_to_diagonal} one gets, under the condition \eqref{eqn:H-theta-condtion}, 
\begin{multline*}
   \int_T^{t_{n(T),N}^\theta} (t_{n(T),N}^\theta-u) H^2(u) du 
   + \sum_{n=n(T)+1}^N \int_{t_{n-1,N}^\theta}^{t_{n,N}^\theta} (t_{n,N}^\theta-u) H^2(u) du \\
   \stackrel{N\to\infty}{\to} \frac{1}{2\theta} \int_T^1 (1-t)^{1-\theta} H^2(t)dt.
\end{multline*}


\subsection{Projection onto the remainder}\label{sec:2nd_chaos}
\medskip

\begin{proof}[Proof of \cref{thm:2nd_chaos}]
Consider
\begin{align*}
    \E | R_{\tau_N^\theta}(Z^T-Z^T_{\tau_N^\theta} )|^2
& = \E \left | \int_0^T\int_{l_N(s)}^s \E_t D_t D_s (Z-Z_{\tau_N^\theta} ) dW_t dW_s \right |^2 \\
& = \int_0^T\int_{l_N(s)}^s \E| \E_t D_t D_s (Z-Z_{\tau_N^\theta} )|^2  dt ds.
\end{align*}

For a.a. $s$ and $t$ with $l_N(s) < t < s < u_N(s)\le T$ it holds that
\begin{align*}
&   D_t D_s (Z^T-Z^T_{\tau_N^\theta} ) \\
& = D_t \brac \sigma(X_s)(\delta_N\vph)(s) \kets \\
& + D_s \brac \sigma(X_t) (\delta_N\vph)(t)\kets\\
& \quad + \int_0^T D_t D_s \brac \sigma(X_u) (\delta_N\vph)(u)\kets  \de W_u \\
& \quad  +  \int_0^T D_t D_s \kla  \sigma(X_u) ( \delta_N\nabla_y\vph)(u) \cdot  D_u L \mer du \\
& =  \sigma(X_s)^2\vph_x(s) + \sigma'(X_s) \sigma(X_s) (\delta_N\vph)(s) \\
& \quad  + \sigma(X_s)
    (\delta_N\nabla_y\vph)(s) \cdot  D_t L \\
& \quad  + \sigma(X_t) (\delta_N\nabla_y\vph)(t) \cdot  D_s L \\
& \quad  +  \int_0^T D_t D_s \brac \sigma(X_u)
    (\delta_N\vph)(u) \kets  \de W_u \\
& \quad  + \int_0^T D_t D_s \kla  \sigma(X_u) (\delta_N \nabla_y \vph)(u) \cdot D_u L \mer du \\
& =: S_N^1(s) + S_N^2(s) + S_N^3(t,s) + S_N^3(s,t)
     + S_N^4(t,s) + S_N^5(t,s).
\end{align*}
\bigskip
{\bf Term $S_N^1(s)$:} Let
      \[ A(t,s) := \E \bet \E_t \sigma^2(X_s) \vph_x(s,X_s,L) \rag^2
            \sptext{1}{for}{1}
            0\le t \le s \le T. \]
      Then
      \[ \lim_{t\uparrow s}  A(t,s) = A(s,s), \]
      where the convergence is from below. Moreover, $A(s,s)$ is left-continuous.
      We have to show that
      \[   \lim_{N \to \infty}N
            \int_0^T \int_{l_N(s)}^s A(t,s) dt ds
            = \frac{1}{2\theta} \int_0^T (1-s)^{1-\theta} A(s,s) ds, \]
      which follows from \cref{lemma:convergence_to_diagonal}.
      \bigskip

{\bf Term $S_N^2(s)$}: Here we get that
\[ \lim_{N\to\infty} \sup_{s\in [0,T]} \| S_N^2(s)\|_{L_2} = 0. \]

{\bf Terms $S_N^3(s,t)$ and $S_N^3(t,s)$}:
      We only consider the first term. Let $2<p,q<\infty$ such that
      $\frac{1}{2} = \frac{1}{p} + \frac{1}{q}$. Then
\begin{align*}
     \E [S_N^3(s,t)]^2
& =  \E [ \sigma(X_t) (\delta_N\nabla_y\vph)(t) \cdot  D_s L]^2\\
&\le \E \left[ \| \sigma(X_t)  (\delta_N\nabla_y\vph)(t) \|^2_{l_q^K} \pl
            \| D_s L\|^2_{l_{q'}^K}\right] \\
&\le \kla \E\left[ \big | \sigma(X_t)\big |^q
            |\delta_N\nabla_y \vph(t)|^q\right]
      \mer^\frac{2}{q}
      \kla \E  \kla \sum_{m=1}^K |D_s L_m|^{q'}\mer^\frac{p}{q'}
      \mer^\frac{2}{p} \\
& =: c_{p,q,L}  \kla \E \left[\big | \sigma(X_t)\big | ^q
            |\delta_N\nabla_y \vph(t)|^q\right]
      \mer^\frac{2}{q}
\end{align*}
and
\[ \lim_{N\to\infty} \sup_{t\in [0,T]}  \E \left[\big | \sigma(X_t)\big | ^q
            |\delta_N\nabla_y \vph(t)|^q\right] = 0.\]

{\bf Term $S_N^4(t,s)$:} 
We use the covariance formula for the Skorohod integral  \cref{prop:skorohod_covariance} and get
\begin{align*}
     \E \klas \E_t  S_N^{4}(t,s) \mers^2
&\le \E |S_N^{4}(t,s)|^2 \\
&\le \E \int_0^T | D_t D_s \brac \sigma (X_u) (\delta_N \vph)(u) \kets |^2  du \\
&\quad  + \E \int_0^T \int_0^T  | D_r  D_t D_s
      \brac \sigma(X_u) (\delta_N \vph)(u) \kets |^2 dudr.
\end{align*}
These terms are handled by \cref{lemma:convergence2} below.
\medskip

{\bf Term $S_N^5(t,s)$:} Here we have that
\begin{align*}
     \E \klas \E_t S_N^5(t,s) \mers^2
&\le \E \klas      S_N^5(t,s) \mers^2 \\
&\le \int_0^T \E \klas D_t D_s \kla  \sigma(X_u)(\delta_N \nabla_y \vph)(u) \cdot
      D_u L \mer \mers^2 du
\end{align*}
and proceed similarly as before because the $\delta_N$-functional is present which
ensures the convergence.
\end{proof}

\begin{lemma}\label{lemma:convergence2}
Let Assumption (A) hold and $\sigma_N^\theta$, $N>M$, be the nets from \cref{def:time_nets}. Then it holds that
\begin{align*}
      \lim_{N\to \infty} N \int_0^T\int_{l_N(s)}^s \int_0^T \E \bet D_t D_s
      \brac (\sigma (X_u)\delta_N \vph)(u) \kets \rag^2  du dt ds
& = 0, \\
       \lim_{N\to \infty} N \int_0^T\int_{l_N(s)}^s \E \int_0^T \int_0^T  \bet D_r  D_t D_s
       \brac (\sigma (X_u)\delta_N \vph) (u) \kets \rag^2 dr du dt ds
& = 0.
\end{align*}
\end{lemma}

\begin{proof}
 By  \eqref{ito}  and \cref{prop:skorohod_IBP} we have
\begin{align*}
\sigma (X_u) \delta_N \vph(u) =&  \int_{l_N(u)}^u \sigma (X_u)\vph_x(r,X_r,L) \sigma(X_r) \delta W_r \notag \\
&+  \int_{l_N(u)}^u \vph_x(r,X_r,L) \sigma(X_r) D_r \sigma (X_u) dr \notag \\
& +  \int_{l_N(u)}^u  \sigma (X_u) \sigma(X_r)\nabla_y \vph_{x}(r,X_r,L)\cdot D_rL dr \notag \\
& - \int_{l_N(u)}^{u} \sigma (X_u) \sigma(X_r)\sigma'(X_r) \vph_x(r,X_r,L) dr.
\end{align*}
 
For  $   l_N(s) <t <s <T$   and $u \in (0,T)$ we get
\begin{align*}
D_tD_s \big ( \sigma (X_u) \delta_N \vph(u)\big ) 
&= D_t ( \sigma (X_u)\vph_x(s,X_s,L) \sigma(X_s) )\one_{(l_N(u),u]}(s)\\
& \quad +   D_s ( \sigma (X_u)\vph_x(t,X_t,L) \sigma(X_t)) \one_{(l_N(u),u]}(t)  \\
& \quad +   \int_{l_N(u)}^u  D_tD_s (\sigma (X_u)\vph_x(r,X_r,L) \sigma(X_r) ) \delta W_r \notag \\
& \quad +  \int_{l_N(u)}^u D_tD_s \delta(u,r)dr \\
&= A_1(t,s,u) + A_2 (t,s,u) + A_3(t,s,u)  + A_4 (t,s,u) 
\end{align*}
with  
\begin{align*}
    \delta(u,r)
&:= \vph_x(r,X_r,L) \sigma(X_r) D_r \sigma (X_u)   +  \sigma (X_u) \sigma(X_r)\nabla_y \vph_{x}(r,X_r,L)\cdot D_rL \\
&\quad  - \sigma (X_u) \sigma(X_r)\sigma'(X_r) \vph_x(r,X_r,L).
\end{align*}
Therefore,
\begin{multline*}
 N \int_0^T\int_{l_N(s)}^s \int_0^T \E | A_1(t,s,u) |^2  du dt ds \\=
N \int_0^T\int_{l_N(s)}^s\int_{l_N(s)}^{u_N(s)} \E | A_1(t,s,u) |^2  du dt ds
\end{multline*}
is of order $1/N$  since $\E | A_1(t,s,u) |^2 $ is uniformly bounded in $(t,s,u)$.    A similar argument holds for $A_2 (t,s,u) $.
   For  $ A_3(t,s,u)$ we first use \cref{prop:skorohod_covariance}    and get integrals over $({l_N(u)},u]$.  
   Since the integrand is uniformly bounded in $L_2(\PP),$  also here we get  convergence to zero.
   Finally, for $A_4 (t,s,u)$ we estimate first use H{\"o}lder's inequality and then repeat the above arguments.
    The second limit equation can be treated in the same way.
\end{proof}


\subsection{Proof of \cref{thm:error_limit}} \label{sec:proof:thm:error_limit} 
By \cref{sec:proof_strategy} we only need to consider $Z^T$.
We have by \cref{first-chaos} and \cref{thm:2nd_chaos} that
\begin{align*}
&   \lim_{N\to\infty} N \var (Z^T-Z^T_{\tau_N^\theta}) \\
& = \frac{1}{2\theta} \int_0^T (1-t)^{1-\theta}
     \int_t^T \noo \E_s [ \sigma(X_t)^2\nabla_y \vph_x(t,X_t,L)
     \cdot D_s L]\rrm^2_{L_2}dsdt \\
& +\frac{1}{2\theta} \int_0^T (1-t)^{1-\theta} \noo \sigma^2(X_t)
    \E_t  \vph_x(t,X_t,L) \rrm^2_{L_2} dt \\
& = \frac{1}{2\theta} \int_0^T (1-t)^{1-\theta}
    \noo \sigma^2(X_t) \vph_x(t,X_t,L)  \rrm^2_{L_2} dt,
\end{align*}
where the last line follows immediately from the  Clark-Ocone formula
\[ \vph_x(t,X_t,L) = \E_t  \vph_x(t,X_t,L)   + \int_t^T \E_s D_s \vph_x(t,X_t,L) dW_s.\]


\begin{appendix}


\section{Malliavin Calculus}

We collect some facts from Malliavin calculus.
General references are \cite{Nualart:Pardoux:1988} and \cite{Nualart:book:2006}.
\medskip

\begin{definition}[Skorohod integral] \
\begin{enumerate}[(i)]
            \item We let $\dom(\delta)$ be the set of $u \in L_2([0,1]\times\Omega)$ such that
            \[
                  \left|\mathbb{E}\int_{0}^{1} D_tF u_tdt\right|^2 \leq C_u \E |F|^2
            \]
            for all $F \in \mathbb{D}_{1,2}$, where $C_u\ge 0$ is a constant depending on $u$.

            \item If $u\in \dom(\delta)$, then $\delta(u)$ is the element of $L_2(\Omega)$ characterized by
            \[
            \mathbb{E}(F\delta(u)) = \mathbb{E}\int_0^1 (D_t F) u_t dt \sptext{1}{for}{1} F \in \mathbb{D}_{1,2}.
            \]
           
      \end{enumerate}
The operator $\delta$ is called the {\normalfont divergence operator} and the divergence $\delta(u)$ is called the {\it Skorohod integral} of the process $u$, denoted by
      \[
            \delta(u) = \int_0^1 u_t \delta W_t.
      \]

\end{definition}

\begin{definition}[Space $\mathbb{L}_{1,2}$]
We let $\mathbb{L}_{1,2}$ be the space of $u \in L_2([0,1]\times\Omega)$
such that there is a version of $u$ with $u_t \in \mathbb{D}_{1,2}$ for all $t\in [0,1]$
and there are representatives of the Malliavin derivatives such that
$(s,t,\omega) \mapsto D_su_t(\omega)$ is
$\mathcal{B}([0,1])\otimes\mathcal{B}([0,1])\otimes\mathcal{F}$-measurable
with $\E \int_0^1\int_0^1(D_su_t)^2dsdt < \infty$.
\end{definition}
\medskip

\begin{proposition}[Covariance formula for the Skorohod integral, {\cite[Proposition 3.1]{Nualart:Zakai:1986}}, {\cite[(1.54) on page 42]{Nualart:book:2006}}]
\label{prop:skorohod_covariance}
One has $\mathbb{L}_{1,2}\subseteq \dom(\delta)$ and for $u,v \in \mathbb{L}_{1,2}$ one has, a.s.,
\[\mathbb{E}\left(\int_0^1 u_t \delta W_t\int_0^1v_t \delta W_t\right) = \int_0^1 \mathbb{E}(u_tv_t)dt + \int_0^1\int_0^1\mathbb{E}(D_su_tD_tv_s)dsdt. \]
\end{proposition}
\bigskip

From the above proposition one gets for $0\le a \le b \le 1$ and
$u\in \mathbb{L}_{1,2}$ that $\one_{[a,b]}u\in \mathbb{L}_{1,2}\subseteq \dom(\delta)$ and can
use the notation
\[ \int_a^b u_t \delta W_t := \delta(\one_{[a,b]}u).\]

\begin{proposition}[{\cite[Theorem 3.2]{Nualart:Pardoux:1988}}]
\label{prop:skorohod_IBP}
Let $F \in \mathbb{D}_{1,2}$ and $u\in \dom(\delta)$ such that
$Fu \in L_2([0,1]\times\Omega)$. Then $Fu\in \dom(\delta)$ and
\[ \int_0^1Fu_t \delta W_t = F \int_0^1 u_t \delta W_t -\int_0^1 D_tF u_t dt
   \mbox{ a.s.} \]
if the right hand side belongs to $L_2(\Omega)$.
\end{proposition}
\bigskip

\begin{lemma}[It\^o's formula]
\label{lemma:ito}
Assume $f: [0,1]\times\R\times\R^{K} \to  \R$ such that the partial derivatives $f_x, f_t, \nabla_y f, f_{xx}, \nabla_y f_x$ exist, are continuous,
and polynomially bounded in $(x,y)$ uniformly in $t$.  Assume that $L=(L_1,\dots,L_K)$ is a  random variable such that  $L_i \in \DD_{1,4}$.
Then for $0\le s < t \le 1$ one has, a.s.,
\begin{align*}
   f(t, X_t,L)
& = f(s, X_s,L) + \int_s^t f_u(u, X_u,L) du \\
&\quad  +\int_s^t f_x(u, X_u,L) \sigma(X_u) \delta W_u \\
&\quad  +\half \int_s^t f_{xx}(u, X_u,L) \sigma^2(X_u) du \\
&\quad  + \int_s^t\nabla_y f_x(u, X_u,L) \sigma(X_u) \cdot D_uL du.
\end{align*}
\end{lemma}

\begin{proof}
For $y\in \R^K$ fixed we have by Itô's formula, a.s.,
\begin{align*}
    f(t, X_t,y)
& = f(s, X_s,y) + \int_s^t f_u(u, X_u,y) du \\
&\quad +\int_s^t f_x(u, X_u,y) \sigma(X_u)  dW_u \\
&\quad +\half \int_s^t f_{xx}(u, X_u,L) \sigma^2(X_u) du
\end{align*}
from which we get by substitution, a.s.,
\begin{align*}
      f(t, X_t,L)
& = f(s, X_s,L) + \int_s^t f_u(u, X_u,L) du \\
& \quad +\int_s^t f_x(u, X_u,y) \sigma(X_u) d W_u \biggr|_{y = L} \\
& \quad +\half \int_s^t f_{xx}(u, X_u,L) \sigma^2(X_u) du.
\end{align*}
The statement follows from \cite[Proposition 4.12]{Nualart:Pardoux:1988}.
\end{proof}
\medskip

\begin{proposition}[{\cite[Proposition 5.5]{Nualart:Pardoux:1988}}]
\label{prop:derivative_of_skorohod_integral}
For $u \in \mathbb{L}_{2,2}$ one has $\int_0^1 u_t \delta W_t \in \mathbb{D}_{1,2}$ with
      \[
            D_t \int_0^1 u_s \delta W_s = u_t + \int_0^1 D_t u_s \delta W_s \mbox{ a.s.}\sptext{1}{for a.a.}{1} t\in [0,1].
      \]
\end{proposition}
\bigskip

\begin{lemma}\label{lemma:expectation_of_skorohod}
For $u \in \mathbb{L}_{1,2}$ and $s\in [0,1]$ one has
\[\E_s \int_0^1 u_r \delta W_ r = \int_0^1 \left [ \one_{[0,s]}(r)\E_s u_r \right ] \delta W_r
   \mbox{ a.s.}\]
\end{lemma}

\begin{proof}
      For the Wiener chaos expansion
      \[u_r = \sum_{n = 0}^{\infty} I_n(f_n(\cdot, r))\]
      it holds by \cite[ Proposition 1.3.7, Lemma 1.2.5]{Nualart:book:2006} that
\[   \E_s \int_0^1u_r \delta W_r
   = \E_s \sum_{n = 0}^\infty I_{n+1}(\tilde{f}_n)
   = \sum_{n=0}^\infty I_{n+1}(\tilde{f}_n\one_{[0,s]}^{\otimes{n+1}}), \]
where

\equa
\tilde{f}_n(t_1, \dots, t_n, t) &=&
\frac{1}{n+1}\big[f_n(t_1, \dots, t_n,t)
\\
&&+\sum_{i = 1}^{n}f_n(t_1, \dots, t_{i-1},t,t_{i+1}, t_n, t_i)\big]
\tion
is the symmetrization of $f_n$ in all its variables. On the other hand, one has
\[   \int_0^s \E_s u_r \delta W_r
   = \int_0^s \sum_{n = 0}^{\infty} I_n(f_n(\cdot, r) \one_{[0,s]}^{\otimes n}(\cdot)\one_{[0,s]}(r)) \delta W_r
   = \sum_{n=0}^\infty I_{n+1}(\tilde{f}_n\one_{[0,s]}^{\otimes{n+1}}), \]
 which proves the assertion.
\end{proof}

\end{appendix}


{\bf Acknowledgements:}
The third author acknowledges he was supported by the Finnish Ministry of Education and Culture’s Pilot for Doctoral Programmes (Pilot project Mathematics of Sensing, Imaging and Modelling)
as well as the Finnish Centre of Excellence in Randomness and Structures.
We also thank Eija Laukkarinen for the fruitful discussions.



\end{document}